\documentclass[12pt]{article}

\usepackage[margin=1in]{geometry}
\usepackage{amscd}
\usepackage{amssymb}
\usepackage{amsmath}
\usepackage{amsthm}
\usepackage{enumerate}
\usepackage{tikz}
\usepackage{wrapfig, framed, caption}
\usepackage{float}
\usetikzlibrary{arrows}
\usepackage[font=small,labelfont=bf]{caption}
\usepackage{xcolor}
\usepackage{mathtools}
\usepackage[colorlinks=true, linkcolor=blue, citecolor=blue,
pagebackref=true]{hyperref}
\usepackage{ulem}
\usepackage{cite}

\allowdisplaybreaks

\newtheorem{theorem}{Theorem}

\newtheorem*{theoremY*}{Theorem Y}

\newtheorem*{theoremAB*}{Theorem AB}

\newtheorem*{linearformsmtp*}{Mass transference principle for linear forms}
\newtheorem{corollary}{Corollary}
\newtheorem*{corollary*}{Corollary}
\newtheorem{proposition}{Proposition}
\newtheorem{lemma}{Lemma}

\newtheorem*{claim*}{Claim}

\theoremstyle{definition}
\newtheorem{definition}{Definition}
\theoremstyle{remark}

\newtheorem*{remark*}{Remark}

\renewcommand{\Bbb}[1]{\mathbb{#1}}
\newcommand{\N}{{\Bbb N}}         % natural numbers

\newcommand{\R}{{\Bbb R}}        % real numbers
\newcommand{\Z}{{\Bbb Z}}         % integer numbers

\newcommand{\cA}{{\cal A}}
\newcommand{\cB}{{\cal B}}

\newcommand{\cH}{{\cal H}}

\newcommand{\cJ}{{\cal J}}
\newcommand{\cK}{{\cal K}}

\newcommand{\ve}{\varepsilon}
\newcommand{\vphi}{\varphi}

\newcommand{\La}{\Lambda}

\newcommand{\bal}{\beta_\alpha}

\newcommand{\btau}{\boldsymbol{\tau}}

\newcommand{\ba}{\mathbf{a}}
\newcommand{\bb}{\mathbf{b}}
\newcommand{\bc}{\mathbf{c}}
\newcommand{\bi}{\mathbf{i}}
\newcommand{\bj}{\mathbf{j}}
\newcommand{\x}{\mathbf{x}}
\newcommand{\y}{\mathbf{y}}
\newcommand{\p}{\mathbf{p}}

\newcommand{\bt}{\mathbf{t}}

\newcommand{\diam}{\mathrm{diam}}
\newcommand{\dist}{\operatorname{dist}}

\DeclarePairedDelimiter{\abs}{\lvert}{\rvert}
\DeclarePairedDelimiter{\set}{\lbrace}{\rbrace}

\DeclarePairedDelimiter{\supn}{\|}{\|}

\DeclareMathOperator{\dimh}{\dim_H}
\DeclareMathOperator{\Mu}{M}

\title{Weighted approximation in higher-dimensional missing digit sets}

\author{Demi Allen \\ (Exeter) \and Benjamin Ward \\ (York)}
\date{\today}
\begin{document}
\frenchspacing
\maketitle
\begin{abstract}
 In this note, we use the mass transference principle for rectangles, recently obtained by Wang and Wu (Math. Ann., 2021), to study the Hausdorff dimension of sets of ``weighted $\Psi$-well-approximable'' points in certain self-similar sets in $\R^d$. Specifically, we investigate weighted $\Psi$-well-approximable points in ``missing digit'' sets in $\R^d$. The sets we consider are natural generalisations of Cantor-type sets in $\R$ to higher dimensions and include, for example, \emph{four corner Cantor sets} (or \emph{Cantor dust}) in the plane with contraction ratio $\frac{1}{n}$ with $n \in \N$.
\end{abstract}

\section{Introduction and motivation}

The work of this current paper is motivated by a question posed in a seminal paper by Mahler \cite{Mahler_1984}; namely, how well can we approximate points in the middle-third Cantor set by:
\begin{enumerate}[(i)]
\itemsep-2pt
\item{rational numbers contained in the Cantor set, or} \label{Mahler part 1}
\item{rational numbers not in the Cantor set?}
\end{enumerate}

The first contribution to this question was arguably made by Weiss \cite{Weiss_2001}, who showed that almost no point in the middle-third Cantor set is \emph{very well approximable} with respect to the natural probability measure on the middle-third Cantor set. Since this initial contribution, numerous authors have contributed to answering these questions, approaching them from many different perspectives. For example, Levesley, Salp, and Velani \cite{LSV_2007} considered triadic approximation in the middle-third Cantor set, different subsets of the first named author, Baker, Chow, and Yu \cite{ABCY, ACY, Baker4} studied dyadic approximation in the middle-third Cantor set, Kristensen \cite{Kristensen_2006} considered approximation of points in the middle-third Cantor set by algebraic numbers, and Tan, Wang and Wu \cite{TWW_2021} have recently studied part (\ref{Mahler part 1}) by introducing a new notion of the ``height'' of a rational number. There has also been considerable effort invested in trying to generalise some of the above results to more general self-similar sets in $\R$ and also to various fractal sets in higher dimensions. See, for example, \cite{Allen-Barany, Baker1, Baker2, Baker3, Baker-Troscheit, BFR_2011, FS_2014, FS_2015, Khalil-Luethi, KLW_2005, PV_2005, WWX_Cantor, Yu_self-similar_2021} and references therein. The results in this paper can be thought of as a contribution to answering a natural $d$-dimensional weighted variation of part (\ref{Mahler part 1}) of Mahler's question. In particular, we will be interested in weighted approximation in $d$-dimensional ``missing digit'' sets. 

 Before we introduce the general framework we will consider here, we provide a very brief overview of some of the classical results on weighted Diophantine approximation in the ``usual'' Euclidean setting which provide further motivation for the current work. Fix $d \in \N$ and let $\Psi=(\psi_{1}, \dots , \psi_{d})$ be a $d$-tuple of approximating functions $\psi_{i}:\N \to [0, \infty)$ with $\psi_{i}(r) \to 0$ as $r \to \infty$ for each $1 \leq i \leq d$. The set of weighted simultaneously $\Psi$-well-approximable points in $\R^d$ is defined as
\begin{equation*}
W_{d}(\Psi):= \left\{ \x = (x_{1}, \dots , x_{d}) \in [0,1]^{d}: \left|x_{i}-\frac{p_{i}}{q}\right| < \psi_{i}(q)\, ,\; 1 \leq i \leq d, \text{ for i.m.} \; (p_1, \dots p_d, q) \in \Z^{d} \times \N \right\},
\end{equation*}
where i.m. denotes infinitely many. Note that the special case  where each approximating function is the same, that is $\Psi=(\psi, \dots , \psi)$, is generally the more intensively studied set. The case where each approximating function is potentially different, usually referred to as \emph{weighted simultaneous approximation}, is a natural generalisation of this. Simultaneous approximation (i.e. when the approximating function is the same in each coordinate axis) can generally be seen as a metric generalisation of Dirichlet's Theorem, whereas weighted simultaneous approximation is a metric generalisation of Minkowski's Theorem. Weighted simultaneous approximation has earned interest in the past few decades due to Schmidt  and natural connections to Littlewood’s Conjecture, see for example \cite{BRV_2016, BPV_2011, An_2013, An_2016, Schmidt_1983}. 
 
Motivated by classical works due to the likes of Khintchine \cite{Khintchine_24, Khintchine_25} and Jarn\'{i}k \cite{Jarnik_31} which tell us, respectively, about the Lebesgue measure and Hausdorff measures of the sets of classical simultaneously $\Psi$-well-approximable points (i.e. when $\Psi=(\psi, \dots, \psi)$), one may naturally also wonder about the ``size'' of sets of weighted simultaneously $\Psi$-well-approximable points in terms of Lebesgue measure, Hausdorff dimension, and Hausdorff measures. Khintchine \cite{Weighted_Khintchine} showed that if $\psi: \N \to [0,\infty)$ and $\Psi(q)=(\psi(q)^{\tau_1}, \dots, \psi(q)^{\tau_d})$ for some $\btau = (\tau_1,\dots,\tau_d) \in (0,1)^d$ with $\tau_1+\tau_2+\dots+\tau_d=1$, then
\begin{equation*}
\lambda_d(W_d(\Psi))=\begin{cases}
\quad 0  \quad  &\text{ if } \quad \sum_{q=1}^{\infty} q^d \psi(q) < \infty, \\[2ex]
\quad 1 \quad &\text{ if } \quad \sum_{q=1}^{\infty} q^d \psi(q) = \infty, \text{ and $q^d\psi(q)$ is monotonic}.
\end{cases}
\end{equation*}
Throughout we use $\lambda_d(X)$ to denote the $d$-dimensional Lebesgue measure of a set \mbox{$X \subset \R^d$}. For more general approximating functions $\Psi(q)=(\psi_1(q), \dots, \psi_d(q))$, with $\prod_{i=1}^{d}\psi_i(q)$ monotonically decreasing and $\psi_{i}(q)<q^{-1}$ for each $1\leq i \leq d$, it has been proved, see \cite{Cassels_1950, Weighted_Khintchine, Schmidt_1960, Gallagher_1962}, that
\begin{equation*}
\lambda_d(W_d(\Psi))=\begin{cases}
\quad 0  \quad  &\text{ if } \quad \sum_{q=1}^{\infty} q^d \psi_1(q)\dots \psi_d(q) < \infty, \\[2ex]
\quad 1 \quad &\text{ if } \quad \sum_{q=1}^{\infty} q^d \psi_1(q)\dots \psi_d(q) = \infty.
\end{cases}
\end{equation*}
For approximating functions of the form $\Psi(q)=(\psi_1(q),\dots,\psi_d(q))$ where
\[\psi_i(q)=q^{-t_{i}-1}, \quad \text{for some vector } \bt=(t_{1}, \dots , t_{d}) \in \R^{d}_{>0},\] 
Rynne \cite{Rynne_1998} proved that if $\sum_{i=1}^{d}t_{i} \geq 1$, then 
\begin{equation*}
\dimh W_{d}(\Psi)=\min_{1 \leq k \leq d} \left\{ \frac{1}{t_{k}+1} \left( d+1+\sum_{i:t_{k} \geq t_{i}}(t_{k}-t_{i})\right) \right\}.
\end{equation*}
 Throughout, we write $\dimh{X}$ to denote the \emph{Hausdorff dimension} of a set $X \subset \R^d$, we refer the reader to \cite{Falconer} for definitions and properties of Hausdorff dimension and Hausdorff measures. Rynne's result has recently been extended to a more general class of approximating functions by Wang and Wu \cite[Theorem 10.2]{WW2019}.

In recent years, there has been rapidly growing interest in whether similar statements can be proved when we intersect $W_d(\Psi)$ with natural subsets of $[0,1]^{d}$, such as submanifolds or fractals. The study of such questions has been further incentivised by many remarkable works of the recent decades, such as \cite{KLW_2005, KM_1998, VV_2006}, and applications to other areas, such as wireless communication theory \cite{ABLVZ_2016}.

\section{$d$-dimensional missing digit sets and main results}

In this paper we study weighted approximation in $d$-dimensional missing digit sets, which are natural extensions of classical missing digit sets (i.e. generalised Cantor sets) in $\R$ to higher dimensions. A very natural class of higher dimensional missing digit sets included within our framework are the \emph{four corner Cantor sets} (or \emph{Cantor dust}) in $\R^2$ with contraction ratio $\frac{1}{n}$ for $n \in \N$.

Throughout we consider $\R^d$ equipped with the supremum norm, which we denote by $\supn{\cdot}$. For subsets $X,Y \subset \R^{d}$ we define $\diam(X) = \sup\{\|u-v\|:u,v \in X\}$ and $\dist(X,Y)= \inf\{\|x-y\|: x \in X, y \in Y\}$. We define \emph{higher-dimensional missing digit sets} via iterated function systems as follows. Let $b \in \N$ be such that $b \geq 3$ and let $J_1,\dots,J_d$ be proper subsets of $\set{0,1,\dots,b-1}$ such that for each $1 \leq i \leq d$, we have
\[N_i:=\#J_i\geq 2.\]
Suppose $J_i=\set{a^{(i)}_1,\dots,a^{(i)}_{N_i}}$. For each $1 \leq i \leq d$, we define the iterated function system 
\[\Phi^i=\left\{f_j:[0,1]\to[0,1]\right\}_{j=1}^{N_i} \quad \text{where} \quad f_j(x)=\frac{x+a^{(i)}_j}{b}.\]
Let $K_i$ be the attractor of $\Phi^i$; that is, $K_i \subset \R$ is the unique non-empty compact set which satisfies
\[K_i=\bigcup_{j=1}^{N_i}{f_j(K_i)}.\]
 We know that such a set exists due to work of Hutchinson \cite{Hutchinson}. Equivalently $K_i$ is the set of $x \in [0,1]$ for which there exists a base $b$ expansion of $x$ consisting only of digits from $J_i$. In view of this, we will also use the notation $K_{b}(J_i)$ to denote this set. For example, in this notation, the classical middle-third Cantor set is precisely the set $K_3(\{0,2\})$. We call the sets $K_b(J_i)$ \emph{missing digit sets} since they consist of numbers with base-$b$ expansions missing specified digits. Note that, for each $1 \leq i \leq d$, the Hausdorff dimension of $K_i$, which we will denote by $\gamma_i$, is given by 
\[\gamma_i = \dimh{K_i} = \frac{\log{N_i}}{\log{b}}.\]

We will be interested in the \emph{higher-dimensional missing digit set}
\[K:=\prod_{i=1}^{d}{K_i}\]
formed by taking the Cartesian product of the sets $K_i$, $1 \leq i \leq d$. As a natural concrete example, we note that the \emph{four corner Cantor set} in $\R^2$ with contraction ratio $\frac{1}{b}$ (with $b \geq 3$ an integer) can be written in our notation as $K_{b}(\{0,b-1\}) \times K_{b}(\{0,b-1\})$. 

We note that $K$ is the attractor of the iterated function system 
\[\Phi=\left\{f_{(j_1,\dots,j_d)}:[0,1]^d \to [0,1]^d, (j_1,\dots,j_d) \in \prod_{i=1}^{d}{J_i}\right\}\]
where
\[f_{(j_1,\dots,j_d)}\begin{pmatrix} x_1 \\ \vdots \\ x_d\end{pmatrix} = \begin{pmatrix}\frac{x_1+j_1}{b} \\ \vdots \\ \frac{x_d+j_d}{b}\end{pmatrix}.\]
Notice that $\Phi$ consists of 
\[N:=\prod_{i=1}^{d}{N_i}\]
maps and so, for convenience, we will write 
\[\Phi = \left\{g_j:[0,1]^d\to[0,1]^d\right\}_{j=1}^{N}\]
where the $g_j$'s are just the maps $f_{(j_1,\dots,j_d)}$ from above written in some order. The Hausdorff dimension of $K$, which we denote by $\gamma$, is
\[\gamma = \dimh{K} = \frac{\log{N}}{\log{b}}.\]

We will write
\[\La = \set{1,2,\dots,N} \qquad \text{and} \qquad \La^*=\bigcup_{n=0}^{\infty}{\La^n}.\]
We write $\bi$ to denote a word in $\La$ or $\La^*$ and we write $|\bi|$ to denote the length of $\bi$. For $\bi \in \Lambda^*$ we will also use the shorthand notation
\[g_{\bi} = g_{i_1} \circ g_{i_2} \circ \dots \circ g_{i_{|\bi|}}.\]
We adopt the convention that $g_{\emptyset}(x)=x$.

Let $\Psi: \La^* \to [0,\infty)$ be an approximating function. For each $x \in K$, we define the set 
\begin{align*} 
W(x,\Psi)=\left\{y \in K: \supn{y-g_{\bi}(x)}<\Psi(\bi) \text{ for infinitely many } \bi \in \La^*\right\}.
\end{align*}
The following theorem is a special case of \cite[Theorem 1.1]{Allen-Barany}. 

\begin{theorem} \label{self-similar approx thm}
Let $\Phi$ and $K$ be as defined above. Let $x \in K$ and let $\vphi: \N \to [0,\infty)$ be a monotonically decreasing function. Let $\Psi(\bi) = \diam(g_{\bi}(K))\vphi(|\bi|)$. Then, for $s>0$, 
\[
\cH^{s}(W(x,\Psi))=
\begin{cases}
0 & \text{if} \quad \sum_{\bi \in \La^*}{\Psi(\bi)^s}<\infty, \\[2ex]
\cH^s(K) & \text{if} \quad \sum_{\bi \in \La^*}{\Psi(\bi)^s}=\infty.
\end{cases}
\]
\end{theorem}
Of particular interest to us here is the following easy corollary.
\begin{corollary} \label{simultaneous corollary}
Let $\Phi$ and $K$ be as above and suppose that $\diam(K)=1$. Let $\psi: \N \to [0,\infty)$  be such that $b^n\psi(b^n)$ is monotonically decreasing and define $\vphi: \N \to [0,\infty)$ by $\vphi(n)=b^n\psi(b^n)$. Let $\Psi(\bi) = \diam(g_{\bi}(K))\vphi(|\bi|)$. Recall that $\gamma = \dimh{K}$. Then, for $x \in K$, we have
\[
\cH^{\gamma}(W(x,\Psi))=
\begin{cases}
0 & \text{if} \quad \sum_{n=1}^{\infty}{(b^n\psi(b^n))^{\gamma}}<\infty, \\[2ex]
\cH^{\gamma}(K) & \text{if} \quad \sum_{n=1}^{\infty}{(b^n\psi(b^n))^{\gamma}}=\infty.
\end{cases}
\]
\end{corollary}

\begin{proof}
It follows from Theorem \ref{self-similar approx thm} that
\[
\cH^{\gamma}(W(x,\Psi))=
\begin{cases}
0 & \text{if} \quad \sum_{\bi \in \La^*}{\Psi(\bi)^{\gamma}}<\infty, \\[2ex]
\cH^{\gamma}(K) & \text{if} \quad \sum_{\bi \in \La^*}{\Psi(\bi)^{\gamma}}=\infty.
\end{cases}
\]
However, in this case, by the definition of $\vphi$ and our assumption that $\diam(K)=1$, we have
\[\sum_{\bi \in \La^*}{\Psi(\bi)^{\gamma}} = \sum_{n=1}^{\infty}{\sum_{\substack{\bi \in \La^* \\ |\bi| = n}}{(\diam(g_{\bi}(K))\vphi(|\bi|))^{\gamma}}} = \sum_{n=1}^{\infty}{\sum_{\substack{\bi \in \La^* \\ |\bi| = n}}{\psi(b^n)^{\gamma}}} = \sum_{n=1}^{\infty}{N^n \psi(b^n)^{\gamma}} = \sum_{n=1}^{\infty}{(b^n \psi(b^n))^{\gamma}}. \qedhere \]
\end{proof}

For an approximating function $\psi: \N \to [0,\infty)$, define
\begin{align} \label{W x psi}
W(x,\psi)=\left\{y \in K: \supn{y-g_{\bi}(x)} < \psi(b^{|\bi|}) \text{ for infinitely many } \bi \in \La^*\right\}.
\end{align}
In essence, $W(x,\psi)$ is a set of ``simultaneously $\psi$-well-approximable'' points in $K$. The following statement regarding these sets can be deduced immediately from Corollary \ref{simultaneous corollary}.

\begin{corollary} \label{W x psi corollary}
Let $\Phi$ and $K$ be defined as above and let $\psi: \N \to [0,\infty)$ be such that $b^n\psi(b^n)$ is monotonically decreasing. Suppose further that $\diam(K)=1$. Then, 
\[
\cH^{\gamma}(W(x,\psi))=
\begin{cases}
0 & \text{if} \quad \sum_{n=1}^{\infty}{(b^n\psi(b^n))^{\gamma}}<\infty, \\[2ex]
\cH^{\gamma}(K) & \text{if} \quad \sum_{n=1}^{\infty}{(b^n\psi(b^n))^{\gamma}}=\infty.
\end{cases}
\]
\end{corollary}

Here we will be interested in weighted versions of the sets $W(x,\psi)$. More specifically, for $\bt=(t_1,\dots,t_d) \in \R^d_{\geq 0}$ and for $x \in K$, we define the \emph{weighted approximation set}
\[W(x,\psi,\bt) = \left\{\y=(y_1,\dots,y_d) \in K: |y_j-g_{\bi}(x)_j|<\psi(b^{|\bi|})^{1+t_i}, 1 \leq j \leq d, \text{ for i.m. } \bi \in \La^*\right\}.\]
Here we are using the notation $g_{\bi}(x)=(g_{\bi}(x)_1,\dots,g_{\bi}(x)_d)$. Our main results relating to the Hausdorff dimension of sets of the form $W(x,\psi,\bt)$ are as follows.

\begin{theorem} \label{lower bound theorem}
Let $\Phi$ and $K$ be defined as above. Recall that $\gamma = \dimh{K}$ and $\gamma_i=\dimh{K_i}$ for each $1 \leq i \leq d$. Let $\psi: \N \to [0, \infty)$ be such that $b^n\psi(b^n)$ is monotonically decreasing. Further suppose that $\diam(K)=1$ and
\[\sum_{n=1}^{\infty}{(b^n\psi(b^n))^{\gamma}} = \infty.\]
Then, for $\bt = (t_1,\dots,t_d) \in \R^d_{\geq 0}$, we have
\[\dimh{W(x,\psi,\bt)} \geq \min_{1 \leq k \leq d}\left\{\frac{1}{1+t_k}\left(\gamma + \sum_{j:t_j \leq t_k}{(t_k-t_j)\gamma_{j}}\right)\right\}.\]
\end{theorem}

If $\psi$ satisfies more stringent divergence conditions, then we an show that the lower bound given in Theorem \ref{lower bound theorem} in fact gives an exact formula for the Hausdorff dimension of $W(x,\psi,\bt)$. More precisely, we are able to show the following.

\begin{theorem} \label{upper bound theorem}
Let $\Phi$ and $K$ be as defined above. Let $x \in K$ and let $\psi:\N \to [0,\infty)$ be such that:
\begin{enumerate}[(i)]
    \item{$b^n\psi(b^n)$ is monotonically decreasing,}
    \item{$\displaystyle{\sum_{n=1}^{\infty}{(b^n \psi(b^n))^{\gamma}}=\infty}, \quad$ and}
    \item{$\displaystyle{\sum_{n=1}^{\infty}{(b^n \psi(b^n)^{1+\ve})^{\gamma}}<\infty} \quad$ for every $\ve>0$.}
\end{enumerate}
Then, for $\bt = (t_1, \dots, t_d) \in \R^d_{\geq 0}$, we have 
\[\dimh{W(x,\psi,\bt)} = \min_{1 \leq k \leq d}\left\{\frac{1}{1+t_k}\left(\gamma + \sum_{j:t_j \leq t_k}{(t_k-t_j)\gamma_{j}}\right)\right\}.\]
\end{theorem}
As an example of an approximating function which satisifies conditions $(i)-(iii)$, one can think of $\psi(q)=\left(q(\log_{b}q)^{1/\gamma}\right)^{-1}$. This function naturally appears when one considers analogues of Dirichlet's theorem in missing digit sets (see \cite{BFR_2011, FS_2014}).
 As a corollary to Theorem \ref{upper bound theorem} we deduce the following statement which can be interpreted as a higher-dimensional weighted generalisation of \cite[Theorem 4]{LSV_2007}. In \cite[Theorem 4]{LSV_2007}, Levesley, Salp, and Velani establish the Hausdorff measure of the set of points in a one-dimensional base-$b$ missing digit set (i.e. of the form $K_b(J)$ in our present notation) which can be well-approximated by rationals with denominators which are powers of $b$. Before we state our corollary, we fix one more piece of notation. Given an approximating function $\psi: \N \to [0,\infty)$, an infinite subset $\cB \subset \N$, and $\bt = (t_1,\dots,t_d) \in \R_{\geq 0}^d$, we define
\[W_{\cB}(\psi,\bt)=\left\{x \in K: \left|x_{i}-\frac{p_i}{q}\right|<\psi(q)^{1+t_i}, 1 \leq i \leq d, \text{ for i.m. } (p_1,\dots,p_d,q) \in \Z^d \times \cB \right\}.\]

\begin{corollary} \label{LSV_equivalent}
Fix $b \in \N$ with $b \geq 3$ and let $\cB=\{b^n: n=0,1,2,\dots\}$. Let $K$ be a  higher dimensional missing digit set as defined above (with base $b$) and write $\gamma=\dimh{K}$. Furthermore, suppose that $\set{0,b-1} \subset J_i$ for every $1 \leq i \leq d$. In particular, this also means that $\diam{K}=1$. Let $\psi: \N \to [0,\infty)$ be an approximating function such that 
\begin{enumerate}[(i)]
\item{$b^{n}\psi(b^{n})$ is monotonically decreasing with $b^{n}\psi(b^{n}) \to 0$ as $n \to \infty$},
\item{$\displaystyle{\sum_{n=1}^{\infty}{(b^n \psi(b^n))^{\gamma}}=\infty}, \quad$ and}
\item{$\displaystyle{\sum_{n=1}^{\infty}{(b^n \psi(b^n)^{1+\ve})^{\gamma}}<\infty} \quad$ for every $\ve>0$.}
\end{enumerate}
Then
\begin{equation*}
\dimh W_{\cB}(\psi, \bt) = \min_{1 \leq k \leq d}\left\{\frac{1}{1+t_k}\left(\gamma + \sum_{j:t_j \leq t_k}{(t_k-t_j)\, \gamma_j}\right)\right\}.
\end{equation*}
\end{corollary}

\begin{proof} 

Observe that the conditions imposed in the statement of Corollary \ref{LSV_equivalent} guarantee that Theorem \ref{upper bound theorem} is applicable. Furthermore, by our assumption that $b^{n}\psi(b^{n}) \to 0$ as $n \to \infty$, we may assume without loss of generality that $\psi(b^n) < b^{-n}$ for all $n \in \N$.

Next, we note that if $\p=(p_1,\dots,p_d) \in \Z^d$ and $\frac{\p}{b^n} = \left(\frac{p_1}{b^n},\dots,\frac{p_d}{b^n}\right) \notin K$, then we must have 
\[\dist\left(\frac{\p}{b^n},K\right) \geq b^{-n}, \quad \text{where} \quad \dist(x,K)=\inf\set{\|x-y\|: y \in K}.\]  
(Recall that we use $\|\cdot\|$ to denote the supremum norm in $\R^d$.)
Thus we need only concern ourselves with pairs $(\p,q) \in \Z^{d} \times \cB$ for which $\frac{\p}{q} \in K$.

Let $G=\left\{x=(x_1,\dots,x_d) \in \{0,1\}^{d}\right\}$ and note that $G \subset K$ by the assumption that $\{0,b-1\} \subset J_{i}$ for each $1\leq i \leq d$. For any $x \in G$ and any $\bj \in \Lambda^{n}$ it is possible to write $g_{\bj}(x)=\frac{\p}{b^{n}}$ for some $\p \in (\N\cup \set{0})^{d}$. Hence
\[W(x, \psi, \bt) \subset W_{\cB}(\psi, \bt).\]
Furthermore, the set of all rational points of the form $\frac{\p}{b^{n}}$ contained in $K$ is
\begin{equation*}
\bigcup_{x \in G} \bigcup_{\bj \in \Lambda^{n}} g_{\bj}(x).
\end{equation*}
Hence
\begin{equation*} 
W_{\cB}(\psi, \bt) \subset \bigcup_{x \in G}  W(x, \psi, \bt).
 \end{equation*}
By the finite stability of Hausdorff dimension (see \cite{Falconer}), Corollary \ref{LSV_equivalent} now follows from Theorem~\ref{upper bound theorem}.
\end{proof}

Notice that in Theorem \ref{lower bound theorem}, Theorem \ref{upper bound theorem}, and Corollary \ref{LSV_equivalent}, we insist on the same underlying base $b$ in each coordinate direction. This is somewhat unsatisfactory and one might hope to be able to obtain results where we can have different bases $b_i$ in each coordinate direction. The first steps towards proving results relating to weighted approximation in this setting can be seen in \cite[Section 12]{WW2019}. Proving more general results with different bases in different coordinate directions is likely to be a very challenging problem since such sets are \emph{self-affine} and, generally speaking, self-affine sets are more difficult to deal with than self-similar or self-conformal sets. Indeed, very little is currently known even regarding non-weighted approximation in self-affine sets.

\noindent{\bf Structure of the paper:} The remainder of the paper will be arranged as follows. In Section~\ref{measure theory section} we will present some measure theoretic preliminaries which will be required for the proofs of our main results. The key tool required for proving Theorem \ref{lower bound theorem} is a mass transference principle for rectangles proved recently by Wang and Wu \cite{WW2019}. We introduce this in Section \ref{mtp section}. In Section \ref{lower bound section} we present our proof of Theorem \ref{lower bound theorem} and we conclude in Section \ref{upper bound section} with the proof of Theorem \ref{upper bound theorem}.

\section{Some Measure Theoretic Preliminaries} \label{measure theory section}

Recall that $\gamma = \dimh{K}$ and that $\gamma_i=\dimh{K_i}$ for $1 \leq i \leq d$, where $K$ and $K_i$ are as defined above. Furthermore, note that $0 < \cH^{\gamma}(K) < \infty$ and $0< \cH^{\gamma_{i}}(K_{i})<\infty$ for each $1 \leq i \leq d$, see for example \cite[Theorem 9.3]{Falconer}. Let us define the measures
\[\mu:=\frac{\cH^{\gamma}|_{K}}{\cH^{\gamma}(K)} \qquad \text{and} \qquad \mu_i:=\frac{\cH^{\gamma_i}|_{K_i}}{\cH^{\gamma_i}(K_i)} \quad \text{for each } 1 \leq i \leq d.\]
So, for $X \subset \R^d$, we have 
\[\mu(X) = \frac{\cH^{\gamma}(X \cap K)}{\cH^{\gamma}(K)}.\]
Similarly, for $X \subset \R$, for each $1 \leq i \leq d$ we have 
\[\mu_i(X) = \frac{\cH^{\gamma_i}(X \cap K_i)}{\cH^{\gamma_i}(K_i)}.\]
Note that $\mu$ defines a probability measure supported on $K$ and, for each $1 \leq i \leq d$, $\mu_i$ defines a probability measure supported on $K_i$. 
Note also that the measure $\mu$ is $\delta$-Ahlfors regular with $\delta = \gamma$ and, for each $1 \leq i \leq d$, the measure $\mu_i$ is $\delta$-Ahlfors regular with $\delta=\gamma_i$ (see, for example, \cite[Theorem 4.14]{Mattila_1999}).

We will also be interested in the product measure
\[\Mu:=\prod_{i=1}^{d}{\mu_i}.\]
We note that $\Mu$ is $\delta$-Ahlfors regular with $\delta = \gamma$. This fact follows straightforwardly from the Ahlfors regularity of each of the $\mu_i$'s.

\begin{lemma} \label{M regularity lemma}
The product measure $\Mu = \prod_{i=1}^{d}{\mu_i}$ on $\R^d$ is $\delta$-Ahlfors regular with $\delta=\gamma$.
\end{lemma}

\begin{proof}
Let $B=\prod_{i=1}^{d}{B(x_i,r)}$, $r>0$, be an arbitrary ball in $\R^d$. The aim is to show that \mbox{$\Mu(B) \asymp r^{\gamma}$}. Recall that for each $1 \leq i \leq d$, the measure $\mu_i$ is $\delta$-Ahlfors regular with \mbox{$\delta = \gamma_i = \dimh{K_i} = \frac{\log{N_i}}{\log{b}}$}. Also recall that $N=\prod_{i=1}^{d}{N_i}$ and $\gamma = \dimh{K} = \frac{\log{N}}{\log{b}}$. Thus, we have
\[\Mu(B) = \prod_{i=1}^{d}{\mu_i(B(x_i,r))} \asymp \prod_{i=1}^{d}{r^{\gamma_i}} = r^{\sum_{i=1}^{d}{\gamma_i}}.\]
Note that
\[\sum_{i=1}^{d}{\gamma_i} = \sum_{i=1}^{d}{\frac{\log{N_i}}{\log{b}}} = \frac{\log(\prod_{i=1}^{d}{N_i})}{\log{b}} = \frac{\log{N}}{\log{b}} = \gamma.\]
Hence, $\Mu(B) \asymp r^{\gamma}$ as claimed.
\end{proof}

We also note that, up to a constant factor, the product measure $\Mu$ is equivalent to the measure $\mu = \frac{\cH^{\gamma}|_K}{\cH^{\gamma}(K)}$.

\begin{lemma}\label{measure equivalence lemma}
Let $\Mu=\prod_{i=1}^{d}{\mu_i}$. Then, up to a constant factor, $\Mu$ is equivalent to $\mu$; i.e. for any Borel set $F \subset \R^d$, we have $\Mu(F) \asymp \mu(F)$.
\end{lemma}

Lemma \ref{measure equivalence lemma} follows immediately upon combining Lemma \ref{M regularity lemma} with \cite[Proposition 2.2 (a) + (b)]{Falconer_techniques}.

In our present setting, where $K$ is a self-similar set with well-separated components, we can actually show the stronger statement that $\mu=\Mu$.

\begin{proposition} \label{measures_are_equal}
The measures $\mu$ and $\Mu$ are equal, i.e. for every Borel set $F \subset \R^d$, we have $\mu(F) = \Mu(F)$.
\end{proposition}

\begin{proof}
For each $1 \leq i \leq d$, there exists a unique Borel probability measure (see, for example, \cite[Theorem 2.8]{Falconer_techniques}) $m_i$ satisfying
\begin{align} \label{measure uniqueness 1}
m_i &= \sum_{j=1}^{N_i}{\frac{1}{N_i}m_i \circ f_j^{-1}}.
\end{align}
Likewise, there exists a unique Borel probability measure $m$ satisfying
\begin{align} \label{measure uniqueness 2}
    m &= \sum_{j=1}^{N}{\frac{1}{N}m \circ g_j^{-1}}.
\end{align}

We begin by showing that $\mu_i$ satisfies \eqref{measure uniqueness 1} for each $1 \leq i \leq d$. Note that $\cH^{\gamma_i}(f_{j_1}(K_i) \cap f_{j_2}(K_i)) = 0$ for any $1\leq j_1,j_2 \leq N_i$ with $j_1 \neq j_2$. Thus, for any Borel set $X \subset \R^d$, , we have
\begin{align*}
    \mu_i(X) &= \frac{1}{\cH^{\gamma_i}(K_i)}\cH^{\gamma_i}(X \cap K_{i}) \\
             &= \frac{1}{\cH^{\gamma_i}(K_i)}\sum_{j=1}^{N_i}{\cH^{\gamma_i}(X \cap f_j(K_i))} \\
             &= \frac{1}{\cH^{\gamma_i}(K_i)}\sum_{j=1}^{N_i}{\cH^{\gamma_i}(f_j(f_j^{-1}(X) \cap K_i))} \\
             &= \frac{1}{\cH^{\gamma_i}(K_i)}\sum_{j=1}^{N_i}{\left(\frac{1}{b}\right)^{\gamma_i}\cH^{\gamma_i}(f_j^{-1}(X) \cap K_i)} \\
             &= \frac{1}{\cH^{\gamma_i}(K_i)}\sum_{j=1}^{N_i}{\frac{1}{N_i}\cH^{\gamma_i}(f_j^{-1}(X) \cap K_i)} \\
             &= \sum_{j=1}^{N_i}{\frac{1}{N_i}\mu_i \circ f_j^{-1}(X)}.
\end{align*}

By an almost identical argument, it can be shown that $\mu$ satisfies $\eqref{measure uniqueness 2}$.

Finally, we show that $\Mu$ also satisfies \eqref{measure uniqueness 2} and, hence, by the uniqueness of solutions to~\eqref{measure uniqueness 2}, we conclude that $\Mu$ must be equal to $\mu$. Since $\mu_i$ satisfies \eqref{measure uniqueness 1} for each $1 \leq i \leq d$, we have 
\begin{align*}
\Mu &= \prod_{i=1}^{d}{\mu_i} \\
    &= \prod_{i=1}^{d}{\left(\sum_{j=1}^{N_i}{\frac{1}{N_i}\mu_i \circ f_j^{-1}}\right)} \\
    &= \sum_{\bj = (j_1, \dots, j_d) \in \prod_{i=1}^{d}{\{1,\dots,N_i\}}}{\frac{1}{N}\prod_{i=1}^{d}{\mu_i \circ f_{j_i}^{-1}}} \\
    &= \sum_{j=1}^{N}{\frac{1}{N}M\circ g_j^{-1}}. \qedhere
\end{align*}
\end{proof}

\section{Mass transference principle for rectangles} \label{mtp section}

To prove Theorem \ref{lower bound theorem}, we will use the mass transference principle for rectangles established recently by Wang and Wu in \cite{WW2019}. The work of Wang and Wu generalises the famous Mass Transference Principle originally proved by Beresnevich and Velani \cite{BV_MTP}. Since its initial discovery in \cite{BV_MTP}, the Mass Transference Principle has found many applications, especially in Diophantine Approximation, and has by now been extended in numerous directions. See \cite{Allen-Beresnevich, Allen-Baker, BV_MTP, BV_Slicing, Koivusalo-Rams, WWX2015, WW2019, Zhong2021} and references therein for further information. Here we shall state the general ``full measure'' mass transference principle from rectangles to rectangles established by Wang and Wu in \cite[Theorem~3.4]{WW2019}.

Fix an integer $d \geq 1$. For each $1 \leq i \leq d$, let $(X,|\cdot|_i,m_i)$ be a bounded locally compact metric space equipped with a $\delta_i$-Ahlfors regular probability measure $m_i$. We consider the product space $(X, |\cdot|, m)$ where
\[X = \prod_{i=1}^{d}{X_i}, \qquad |\cdot|=\max_{1 \leq i \leq d}{|\cdot|_i}, \quad \text{and} \quad m=\prod_{i=1}^{d}{m_i}.\]
Note that a ball $B(x,r)$ in $X$ is the product of balls in $\{X_i\}_{1 \leq i \leq d}$;
\[B(x,r) = \prod_{i=1}^{d}{B(x_i,r)} \quad \text{for} \quad x=(x_1,\dots,x_d).\]

Let $J$ be an infinite countable index set and let $\beta: J \to \R_{\geq 0}: \alpha \mapsto \beta_{\alpha}$ be a positive function such that for any $M > 1$, the set
\[\{\alpha \in J: \beta_{\alpha} < M\}\]
is finite. Let $\rho: \R_{\geq 0} \to \R_{\geq 0}$ be a non-increasing function such that $\rho(u) \to 0$ as $u \to \infty$.

For each $1 \leq i \leq d$, let $\{R_{\alpha,i}: \alpha \in J\}$ be a sequence of subsets of $X_i$. Then, the \emph{resonant sets} in $X$ that we will be concerned with are
\[\left\{R_{\alpha}=\prod_{i=1}^{d}{R_{\alpha,i}:\alpha \in J}\right\}.\]
For a vector $\ba = (a_1,\dots,a_d) \in \R_{>0}^d$, write
\[\Delta(R_{\alpha}, \rho(\beta_{\alpha})^{\ba}) = \prod_{i=1}^{d}{\Delta(R_{\alpha,i},\rho(\beta_{\alpha})^{a_i})},\]
where $\Delta(R_{\alpha,i},\rho(\bal)^{a_i})$ appearing on the right-hand side denotes the $\rho(\bal)^{a_i}$-neighbourhood of $R_{\alpha,i}$ in $X_i$. We call $\Delta(R_{\alpha,i},\rho(\beta_{\alpha})^{a_i})$ the \emph{part of $\Delta(R_{\alpha},\rho(\beta_{\alpha})^{\ba})$ in the $i$th direction.}

Fix $\ba = (a_1, \dots, a_d) \in \R_{>0}^d$ and suppose $\bt = (t_1,\dots,t_d) \in \R_{\geq 0}^d$. We are interested in the set
\[W_{\ba}(\bt) = \left\{x \in X: x \in \Delta(R_{\alpha},\rho(\beta_{\alpha})^{\ba+\bt}) \quad \text{for i.m. } \alpha \in J \right\}.\]
We can think of $\Delta(R_{\alpha},\rho(\beta_{\alpha})^{\ba+\bt})$ as a smaller ``rectangle'' obtained by shrinking the ``rectangle'' $\Delta(R_{\alpha},\rho(\beta_{\alpha})^{\ba})$.

Finally, we require that the resonant sets satisfy a certain \emph{$\kappa$-scaling} property, which in essence ensures that locally our sets behave like affine subspaces.

\begin{definition} \label{kappa scaling}
Let $0 \leq \kappa < 1$. For each $1 \leq i \leq d$, we say that $\{R_{\alpha,i}\}_{\alpha \in J}$ has the \emph{$\kappa$-scaling property} if for any $\alpha \in J$ and any ball $B(x,r)$ in $X_i$ with centre $x_i \in R_{\alpha,i}$ and radius $r>0$, for any $ 0 < \varepsilon < r$, we have
\[c_1 r^{\delta_i \kappa}\varepsilon^{\delta_i(1-\kappa
)} \leq m_i(B(x_i,r) \cap \Delta(R_{\alpha,i},\varepsilon)) \leq c_2 r^{\delta_i \kappa} \varepsilon^{\delta_i(1-\kappa)}\]
for some absolute constants $c_1, c_2 > 0$.
\end{definition}
In our case $\kappa=0$ since our resonant sets are points. For justification of this, and calculations of $\kappa$ for other resonant sets, see \cite{Allen-Baker}. Wang and Wu established the following mass transference principle for rectangles in \cite{WW2019}.

\begin{theorem}[Wang -- Wu, \cite{WW2019}] \label{Theorem WW}
Assume that for each $1 \leq i \leq d$, the measure $m_i$ is $\delta_i$-Ahlfors regular and that the resonant set $R_{\alpha,i}$ has the $\kappa$-scaling property for $\alpha \in J$. Suppose
\[m\left(\limsup_{\substack{\alpha \in J \\ \beta_{\alpha} \to \infty}}{\Delta(R_{\alpha},\rho(\beta_{\alpha})^{\ba}})\right) = m(X).\]
Then we have
\[\dimh{W_{\ba}(\bt)} \geq s(\bt):= \min_{A \in \cA}\left\{\sum_{k \in \cK_1}{\delta_k} + \sum_{k \in \cK_2}{\delta_k} + \kappa \sum_{k \in \cK_3}{\delta_k}+(1-\kappa)\frac{\sum_{k \in \cK_3}{a_k \delta_k} - \sum_{k \in \cK_2}{t_k \delta_k}}{A}\right\},\]
where
\[\cA = \{a_i, a_i+t_i: 1 \leq i \leq d\}\]
and for each $A \in \cA$, the sets $\cK_1, \cK_2, \cK_3$ are defined as
\[\cK_1 = \{k: a_k \geq A\}, \quad \cK_2=\{k: a_k+t_k \leq A\} \setminus \cK_1, \quad \cK_3=\{1,\dots,d\}\setminus (\cK_1 \cup \cK_2)\]
and thus give a partition of $\{1,\dots,d\}$.
\end{theorem}

\section{Proof of Theorem \ref{lower bound theorem}} \label{lower bound section}

To prove Theorem \ref{lower bound theorem}, we will apply Theorem \ref{Theorem WW} with $X_i = K_i$, $m_i=\mu_i$ and $\abs{\cdot}_i = \abs{\cdot}$ (absolute value in $\R$) for each $1 \leq i \leq d$. Then, in our setting, we will be interested in the product space $(X, \supn{\cdot}, \Mu)$ where
\[X = \prod_{i=1}^{d}{K_i} \; = K, \qquad \Mu = \prod_{i=1}^{d}{\mu_i},\]
and $\supn{\cdot}$ denotes the supremum norm in $\R^d$. Recall that for each $1 \leq i \leq d$, the measure $\mu_i$ is $\delta_i$-Ahlfors regular with 
\[\delta_i = \gamma_i = \dimh{K_i}\]
and the measure $\Mu$ is $\delta$-Ahlfors regular with
\[\delta = \gamma = \dimh{K}.\]

For us, the appropriate indexing set is 
\[\cJ = \{\bi \in \La^*\}.\]
We define our \emph{weight function} $\beta: \La^* \to \R_{\geq 0}$ by
\[\beta_{|\bi|} = \beta(\bi) = |\bi|.\]
Note that $\beta$ satisfies the requirement that for any real number $M > 1$ the set $\set{\bi \in \La^*: \beta_{\bi} < M}$ is finite. Next we define $\rho: \R_{\geq 0} \to \R_{\geq 0}$ by
\[\rho(u) = \psi(b^u).\]
Since $b^n\psi(b^n)$ is monotonically decreasing by assumption, it follows that $\psi(b^n)$ is monotonically decreasing and $\psi(b^n) \to 0$ as $n \to \infty$.

For a fixed $x = (x_1,\dots,x_d) \in K$, we define the resonant sets of interest as follows. For each $\bi \in \La^*$, take
\[R_{\bi}^x=g_{\bi}(x).\]
Correspondingly, for each $1 \leq j \leq d$, 
\[R_{\bi,j}^x = g_{\bi}(x)_j,\]
where $g_{\bi}(x) = (g_{\bi}(x)_1, \dots, g_{\bi}(x)_d)$. So, $R_{\bi,j}^x$ is the coordinate of $g_{\bi}(x)$ in the $j$th direction. In each coordinate direction, the $\kappa$-scaling property is satisfied with $\kappa$=0, since our resonant sets are points.

Let us fix $\ba = (1,1,\dots,1) \in \R_{>0}^d$. Then, in this case, we note that 
\[\limsup_{\substack{\alpha \in \cJ \\ \bal \to \infty}}{\Delta(R_{\alpha}^{x}, \rho(\bal)^{\ba})} = \limsup_{\substack{\bi \in \La^* \\ |\bi| \to \infty}}{\Delta(g_{\bi}(x), \psi(b^{|\bi|})^{\ba})} = W(x,\psi),\]
where $W(x,\psi)$ is as defined in \eqref{W x psi}. Moreover, it follows from Corollary \ref{W x psi corollary} and Proposition \ref{measures_are_equal} that $\Mu(W(x,\psi)) = \Mu(K)$, since we assumed that $\sum_{n=1}^{\infty}{(b^n\psi(b^n))^{\gamma}} = \infty$.

Now suppose that $\bt = (t_1,\dots,t_d) \in \R_{\geq 0}^d$. Then, in our case, 
\[W_{\ba}(\bt) = W(x,\psi,\bt),\]
which is the set we are interested in. So, recalling that $\kappa = 0$ in our setting, we may now apply Theorem \ref{Theorem WW} directly to conclude that
\[\dimh{W(x,\psi,\bt)} \geq \min_{A \in \cA}\left\{\sum_{k \in \cK_1}{\delta_k} + \sum_{k \in \cK_2}{\delta_k}+\frac{\sum_{k \in \cK_3}{\delta_k}-\sum_{k \in \cK_2}{t_k \delta_k}}{A}\right\} =: s(\bt),\]
where
\[\cA = \set{1} \cup \set{1+t_i: 1 \leq i \leq d}\]
and for each $A \in \cA$ the sets $\cK_1, \cK_2, \cK_3$ are defined as follows:
\[\cK_1 = \set{k: 1 \geq A}, \quad \cK_2 = \{k: 1+t_k \leq A\} \setminus \cK_1, \quad \text{and} \quad \cK_3 = \set{1,\dots,d} \setminus (\cK_1 \cup \cK_2). \]
Note that $\cK_1, \cK_2, \cK_3$ give a partition of $\set{1,\dots,d}$.

To obtain a neater expression for $s(\bt)$, as given in the statement of Theorem \ref{lower bound theorem}, we consider the possible cases which may arise. To this end, let us suppose, without loss of generality, that 
\[0 < t_{i_1} \leq t_{i_2} \leq \dots \leq t_{i_d}.\]

\smallskip
\underline{{\bf Case 1: $A=1$}}

If $A = 1$, then $\cK_1 = \set{1,\dots,d}$, $\cK_2 = \emptyset$, and $\cK_3 = \emptyset$. In this case, the ``dimension number'' simplifies to
\[\sum_{j=1}^{d}{\delta_{j}} = \sum_{j=1}^{d}{\dimh{K_j}} = \sum_{j=1}^{d}{\frac{\log{N_j}}{\log{b}}} = \frac{\log{\left(\prod_{j=1}^{d}N_{j}\right)}}{\log{b}} = \frac{\log{N}}{\log{b}} = \dimh{K}.\]

\smallskip
\underline{{\bf Case 2:} $A=1+t_{i_k}$ with $t_{i_k}>0$}
\nopagebreak

Suppose $A= 1+t_{i_k}$ for some $1 \leq k \leq d$ and that $t_{i_k}>0$ (otherwise we are in Case 1). Suppose $k \leq k' \leq d$ is the maximal index such that $t_{i_k} = t_{i_{k'}}$. In this case, 
\[\cK_1 = \emptyset, \qquad \cK_2 = \set{i_1, \dots, i_{k'}}, \quad \text{and} \quad \cK_3 = \set{i_{k'+1},\dots,i_d}\]
and the ``dimension number'' is
\begin{align*}
    \sum_{j=1}^{k'}{\delta_{i_j}} + \frac{\sum_{j=k'+1}^{d}{\delta_{i_j}}-\sum_{j=1}^{k'}{t_{i_j}\delta_{i_j}}}{1+t_{i_k}} &= \frac{1}{1+t_{i_k}}\left((1+t_{i_k})\sum_{j=1}^{k'}{\delta_{i_j}} + \sum_{j=k'+1}^{d}{\delta_{i_j}}-\sum_{j=1}^{k'}{t_{i_j}\delta_{i_j}}\right) \\
          &= \frac{1}{1+t_{i_k}}\left(\sum_{j=1}^{d}{\delta_{i_j}}+\sum_{j=1}^{k'}{\delta_{i_j}(t_{i_k}-t_{i_j})}\right) \\
          &= \frac{1}{1+t_{i_k}}\left(\dimh{K} + \sum_{j=1}^{k'}{(t_{i_k}-t_{i_j})\dimh{K_j}}\right).
\end{align*}

Putting the two cases together, we conclude that 
\[\dimh{W(x,\psi,\bt)} \geq \min_{1 \leq k \leq d}\left\{\frac{1}{1+t_k}\left(\gamma + \sum_{j:t_j \leq t_k}{(t_k-t_j)\gamma_{j}}\right)\right\},\]
as claimed. This completes the proof of Theorem \ref{lower bound theorem}.

\section{Proof of Theorem \ref{upper bound theorem}} \label{upper bound section}
Let 
\begin{equation*}
\cA_{n}(x, \psi, \bt):= \bigcup_{\bi \in \Lambda^{n}} \Delta\left(R^{x}_{\bi}, \psi(b^{n})^{1+\bt}\right)=\bigcup_{\bi \in \Lambda^{n}}\prod_{j=1}^{d}B\left( R^{x}_{\bi,j}, \psi(b^{n})^{1+t_{j}} \right).
\end{equation*}
Then
\begin{equation*}
W(x, \psi, \bt) = \limsup_{n \to \infty} \cA_{n}(x, \psi, \bt) \, .
\end{equation*}
For any $m \in \N$ we have that 
\begin{equation} \label{cover_1}
W(x, \psi, \bt) \subset \bigcup_{n \geq m} \cA_{n}(x, \psi, \bt) \, .
\end{equation}
Observe that $\cA_{n}(x, \psi, \bt)$ is a collection of $N^{n}=(b^{n})^{\gamma}$ rectangles with sidelengths $2\psi(b^{n})^{1+t_{j}}$ in each $j$th coordinate axis.

Fix some $1 \leq k \leq d$. Throughout suppose that $n$ is sufficiently large such that $\psi(b^{n})<1$. Condition $(i)$ of Theorem \ref{upper bound theorem} implies that $\psi(b^{n})^{1+t_{k}} \leq \psi(b^{n}) \to 0$ as $n \to \infty$, and so for any $\rho>0$ there exists a sufficiently large positive integer $n_{0}(\rho)$ such that 
\begin{equation*}
\psi(b^{n})^{1+t_{k}} \leq \rho \quad \text{ for all } n \geq n_{0}(\rho).
\end{equation*}
 Suppose $n \geq n_0(\rho)$ and that for each $1 \leq j \leq d$ we can construct an efficient finite $\psi(b^{n})^{1+t_{k}}$-cover $\cB_{j}(\bi,k,\rho)$ for $B\left(R^{x}_{\bi,j}, \psi(b^{n})^{1+t_{j}}\right)$ with cardinality $\#\cB_{j}(\bi,k,\rho)$ for each $\bi \in \Lambda^{n}$. Then we can construct a $\psi(b^{n})^{1+t_{k}}$-cover of
$\Delta\left(R^{x}_{\bi}, \psi(b^{n})^{1+\bt}\right)$ for each $\bi \in \Lambda^{n}$ with cardinality $\prod_{j=1}^{d}\#\cB_{j}(\bi,k,\rho)$ by considering the Cartesian product of the individual covers $\cB_j(\bi, k, \rho)$ for each $1 \leq j \leq d$. By \eqref{cover_1}
\begin{equation} \label{cover_2}
 \bigcup_{n \geq n_{0}(\rho)} \cA_{n}(x, \psi, \bt)
\end{equation}
is a cover of $W(x, \psi, \bt)$. So, supposing that we can find such covers $\cB_{j}(\bi,k,\rho)$, we have that
\begin{equation*}
\bigcup_{n \geq n_{0}(\rho)}\bigcup_{\bi \in \Lambda^{n}}\prod_{j=1}^{d}\cB_{j}(\bi,k,\rho)
\end{equation*}
is a $\psi(b^{n})^{1+t_{k}}$-cover of $W(x,\psi,\bt)$.
 
To calculate the values $\#\cB_{j}(\bi,k,\rho)$ we consider two possible cases depending on the fixed $1\leq k\leq d$. Without loss of generality suppose that $0<t_{1}\leq t_{2} \leq \dots \leq t_{d}$. Then, since we are assuming that $\psi(b^{n})<1$, we have that $\psi(b^{n})^{1+t_{1}}\geq \dots \geq \psi(b^{n})^{1+t_{d}}$. 

\smallskip
\underline{{\bf Case 1: $t_{j} \geq t_{k}$}}
 
In this case, $\psi(b^{n})^{1+t_{k}} \geq \psi(b^{n})^{1+t_{j}}$ and so, for any $\bi \in \Lambda^{n}$, we have
\begin{equation*}
B\left(R^{x}_{\bi,j}, \psi(b^{n})^{1+t_{k}}\right) \supset B\left(R^{x}_{\bi,j}, \psi(b^{|\bi|})^{1+t_{j}}\right).
\end{equation*}
Hence, we may take our covers to be $\cB(\bi,k, \rho)=B\left(R^{x}_{\bi,j}, \psi(b^{n})^{1+t_{k}}\right)$, and so $\#\cB_{j}(\bi,k,\rho)=1$.

\smallskip
\underline{{\bf Case 2: $t_{j} < t_{k}$}}

In this case, $\psi(b^{n})^{1+t_{k}} < 2\psi(b^{n})^{1+t_{j}}$. Let $u \in \N$ be the unique integer such that
\begin{equation} \label{u_bound}
b^{-u} \leq 2\psi(b^{n})^{1+t_{j}} < b^{-u+1},
\end{equation}
and observe that, for any $\bi \in \Lambda^{n}$, we have 
\begin{equation*}
B\left(R^{x}_{\bi,j}, \psi(b^{|\bi|})^{1+t_{j}}\right) \subset \bigcup_{\substack{\ba = (a_{1},\dots,a_{u-1}) \in \Lambda_{j}^{u-1} \\ f_{a_i} \in \Phi^{j},\, 1 \leq i \leq u-1}} f_{\ba}([0,1]),
\end{equation*}
where $\Lambda_j = \{1,\dots,N_j\}$. Let $A$ denote the set of $\ba \in \Lambda_{j}^{u-1}$ such that
\begin{equation*}
 f_{\ba}([0,1]) \cap B\left(R^{x}_{\bi,j}, \psi(b^{n})^{1+t_{j}}\right)\neq \emptyset.
 \end{equation*}
Note by the definition of $u$, and the fact that the mappings $f_{\ba}$ of the same length are pairwise disjoint up to possibly a single point of intersection, that $\#A \leq 2$ since 
\begin{equation*}
\diam \left(f_{\ba}([0,1])\right)=b^{-(u-1)} > \diam \left( B\left(R^{x}_{\bi,j}, \psi(b^{n})^{1+t_{j}}\right) \right).
\end{equation*}
 Observe that $f_{\bb}([0,1]) \subset f_{\ba}([0,1])$ if and only if $\bb=\ba\bc$ for $\bc \in \Lambda^{*}_{j}:=\bigcup_{n=0}^{\infty}{\Lambda_j^n}$, where we write $\ba\bc$ to denote the concatenation of the two words $\ba$ and $\bc$. Let $v \geq 0$ be the unique integer such that 
\begin{equation} \label{v_bound}
 b^{-u-v} \leq \psi(b^{n})^{1+t_{k}} < b^{-u-v+1}.
 \end{equation}
Note that $v$ is well defined since $\psi(b^{n})^{1+t_{k}} < 2\psi(b^{n})^{1+t_{j}} < b^{-u+1}$, and so $v \geq 0$. Then 
\begin{equation*}
\underset{\bc \in \Lambda_{j}^{v}}{\bigcup_{\ba \in A, } }f_{\ba \bc}([0,1]) \supset B\left(R^{x}_{\bi,j}, \psi(b^{|\bi|})^{1+t_{j}}\right).
\end{equation*}
Notice that the left-hand side above gives rise to a $\psi(b^n)^{1+t_k}$-cover for the right-hand side and let us denote this cover by $\cB_{j}(\bi,k,\rho)$. By the above arguments an easy upper bound on $\#\cB_{j}(\bi,k,\rho)$ is seen to be $2N_{j}^{v}$. Furthermore, by \eqref{u_bound} and \eqref{v_bound} we have that
\begin{equation*}
\#\cB_{j}(\bi,k,\rho) \leq 2N_{j}^{v} = 2 (b^{v})^{\gamma_{j}} \overset{\eqref{v_bound}}{\leq} 2\left(b^{1-u}\psi(b^{n})^{-1-t_{k}}\right)^{\gamma_{j}} \overset{\eqref{u_bound}}{\leq}2^{1+\gamma_j}b^{\gamma_{j}}\psi(b^{n})^{(t_{j}-t_{k})\gamma_{j}}.
\end{equation*}

Summing over $1 \leq j \leq d$ and $\bi \in \Lambda^n$ for each $n \geq n_0(\rho)$ we see that
 \begin{align}
\cH^{s}_{\rho}(W(x, \psi, \bt)) \, & \ll \sum_{n \geq n_0(\rho)}{\left(\left(\psi(b^n)^{1+t_k}\right)^s \times \sum_{\bi \in \Lambda^n}{\prod_{j=1}^{d}{\#\cB_j(\bi,k,\rho)}}\right)} \nonumber \\ 
&\ll \sum_{n\geq n_{0}(\rho)} \left(\psi(b^{n})^{1+t_{k}}\right)^{s}N^{n} \prod_{j:t_{j}<t_{k}} b^{\gamma_{j}}\psi(b^{n})^{(t_{j}-t_{k})\gamma_{j}} \nonumber \\
& \ll \sum_{n\geq n_{0}(\rho)} \psi(b^{n})^{s(1+t_{k})+\sum_{j:t_{j}<t_{k}}(t_{j}-t_{k})\gamma_{j} - \gamma} \left( \psi(b^{n})b^{n} \right)^{\gamma}.
\end{align}
Thus, it follows from condition $(iii)$ in Theorem \ref{upper bound theorem} that for any
\[s \geq s_{0}=\frac{\gamma+\sum_{j:t_j<t_k}(t_{k}-t_{j})\gamma_{j}+\delta\gamma}{1+t_{k}} \quad \text{with } \delta>0,\] 
we have
\[\cH^{s}_{\rho}(W(x,\psi,\bt)) \to 0 \quad \text{as } \rho \to 0.\]
 This implies that $\dimh W(x, \psi, \bt) \leq s_{0}$. The above argument holds for any initial choice of~$k$, and so we conclude that
\begin{equation*}
\dimh W(x, \psi, \bt) \leq \min_{1 \leq k \leq d} \left\{ \frac{1}{1+t_{k}}\left(\gamma+ \sum_{j:t_j<t_k}(t_{k}-t_{j})\gamma_{j}\right) \right\}.
\end{equation*}
Combining this upper bound with the lower bound result from Theorem \ref{lower bound theorem} completes the proof of Theorem \ref{upper bound theorem}.

\smallskip
\noindent{\bf Acknowledgements.} The authors are grateful to Bal\'{a}zs B\'{a}r\'{a}ny, Victor Beresnevich, Jason Levesley, Baowei Wang, and Wenmin Zhong for useful discussions.

\bibliographystyle{plain}

\begin{minipage}{0.5\linewidth}
  \begin{flushleft} {\footnotesize
      D.~Allen \\
      Coll. of Eng., Maths. and Phys. Sci. \\
      University of Exeter \\
      Harrison Building \\
      North Park Road \\
      Exeter \\ 
      EX4 4QF, UK \\
      \texttt{d.d.allen@exeter.ac.uk} }
  \end{flushleft}
\end{minipage}
\begin{minipage}{0.5\linewidth}
  \begin{flushleft} {\footnotesize
      B.~Ward \\
      Department of Mathematics \\
      University of York \\
      Heslington \\
      York \\ YO10 5DD, UK \\
      \texttt{benjamin.ward@york.ac.uk} }
  \end{flushleft}
\end{minipage}

\end{document}